\documentclass{amsart}
\usepackage{microtype}
\usepackage{pdfsync}
\usepackage{amssymb}            
\usepackage{graphicx}
\usepackage{amscd} 
\usepackage[all]{xy}
\usepackage{hyperref}

\long\def\Notes#1\endNotes{\begin{small}#1\end{small}}

\long\def\Notes#1\endNotes{}

\newcommand{\C}{{\mathbb C}}
\renewcommand{\H}{{\mathbb H}}
\newcommand{\Z}{{\mathbb Z}}
\newcommand{\R}{{\mathbb R}}

\newcommand{\Q}{{\mathbb Q}}

\newcommand{\calC}{{\mathcal C}}
\newcommand{\head}[1]{\iota #1}
\newcommand{\tail}[1]{\tau #1}
\newcommand{\NAH}{NAH}
\newcommand{\I}{\mathcal I}

\newcommand{\Isom}{\operatorname{Isom}}
\newcommand{\wt}{\widetilde}

\newtheorem{theorem}{Theorem}[section]
\newtheorem*{theorem*}{Theorem}
\newtheorem{conjecture}[theorem]{Conjecture}
\newtheorem{ccconjecture}[theorem]{Cusp Covering Conjecture}

\newtheorem{proposition}[theorem]{Proposition}

\newtheorem*{th1}{Classification Theorem}
\newtheorem*{th2}{Realization Theorem}
\newtheorem*{th3}{Commensurablity Theorem}

\theoremstyle{definition}
\newtheorem{definition}[theorem]{Definition}
\newtheorem{property}[theorem]{Property}

\newtheorem{remark}[theorem]{Remark}

\begin{document}
\title 
{Quasi-isometric classification of non-geometric 
3-manifold groups}
\author
{Jason A. Behrstock} \address{Department of Mathematics\\ Lehman 
College, CUNY}
\email{\href{mailto:jason@math.columbia.edu}{jason.behrstock@lehman.cuny.edu}}
\author
{Walter D.  Neumann} \thanks{Research supported under NSF grant no.\
  DMS-0604524} \address{Department of Mathematics\\Barnard College,
  Columbia University}
\email{\href{mailto:neumann@math.columbia.edu}{neumann@math.columbia.edu}}
\begin{abstract}
  We describe the quasi-isometric classification of fundamental groups
  of irreducible non-geometric 3-manifolds which do not have ``too
  many'' arithmetic hyperbolic geometric components, thus completing
  the quasi-isometric classification of $3$--manifold groups in all
  but a few exceptional cases.
\end{abstract}
\maketitle

\section{Introduction}

In \cite{behrstock-neumann} we discussed the quasi-isometry
classification of fundamental groups of $3$--manifolds (which
coincides with the bilipschitz classification of the universal covers
of three-manifolds). This classification reduces easily to the case of
irreducible manifolds. Moreover, no generality is lost by considering
only orientable manifolds. So from now on we only consider compact
connected orientable irreducible $3$--manifolds of zero Euler
characteristic (i.e., with boundary consisting only of tori) since
these are the orientable manifolds which, by Perelman's 
Geometrization Theorem 
\cite{Perelman:Geom1,Perelman:Geom2,Perelman:Geom3},
decompose along tori and Klein bottles into geometric pieces (this
decomposition removes the boundary tori and a family of embedded tori,
so the pieces of the decomposition are without boundary). The minimal
such decomposition is what is called the \emph{geometric decomposition}.

We described (loc.~cit.) the classification for geometric
$3$--manifolds as well as for non-geometric $3$--manifolds with no
hyperbolic pieces in their geometric decomposition (i.e.,
graph-manifolds).  For geometric manifolds this was a summary of work
of others; our contribution was in the non-geometric case. In this
paper we extend to allow hyperbolic pieces. However, our results are
still not quite complete: at present we exclude manifolds with ``too
many'' arithmetic hyperbolic pieces\footnote{The essential remaining
  case to address is when \emph{all} pieces are arithmetic
  hyperbolic. These behave rather differently from the other
  cases---more like arithmetic hyperbolic manifolds.}, and some of our
results are only proved assuming the ``cusp covering conjecture'' in
dimension 3 (see below and Section \ref{sec:ccc}).

In the bulk of this paper we restrict to non-geometric manifolds, all
of whose geometric components are hyperbolic and at least one of which
is non-arithmetic (we will call these \emph{\NAH-manifolds} for
short).  In the final section \ref{sec:seif} we extend to the case
where Seifert fibered pieces are also allowed.

The classification for graph-manifolds in \cite{behrstock-neumann} was
in terms of finite labelled graphs; the labelling consisted of a color
black or white on each vertex and the classifying objects were such
two-colored graphs which are minimal under a relation called
\emph{bisimilarity}.  For \NAH-manifolds the classification is again in
terms of finite labelled graphs, and the the classifying objects are
again given by labelled graphs that are minimal in a similar
sense. The labelling is more complex: each vertex is labelled by the
isomorphism type of a hyperbolic orbifold\footnote{``Hyperbolic
  orbifold'' always means an orientable complete hyperbolic
  $3$--orbifold of finite volume.} and each edge is labelled by a
linear isomorphism between certain 2-dimensional $\Q$--vectorspaces.
We will call these graphs \emph{\NAH-graphs}.  The ones that classify
are the ones that are \emph{minimal} and  \emph{balanced}. All these concepts
will be defined in Section \ref{sec:NAHgraph}.

Finally, when both Seifert fibered and hyperbolic pieces occur the
classifying graphs are a hybrid of the two-colored graphs and the
\NAH-graphs, so we call them \emph{H-graphs}.  In this
case we need that every component of the manifold obtained by removing
all Seifert fibered pieces is \NAH. We will call irreducible
non-geometric manifolds of this type \emph{good}. For example, 
$3$--manifolds which contain a hyperbolic piece but 
no arithmetic hyperbolic piece provide a large family of good 
manifolds.

The following three theorems summarize our main results. 
The second two theorems complement the first, by making 
the classification effective, and then relating  
the quasi-isometric and commensurability classication for some 
\NAH-manifolds; both of these latter two results 
currently need CCC$_3$: the Cusp Covering Conjecture (Conjecture
\ref{cuspcoveringconj}) in dimension $3$.

\begin{th1}
  Each good $3$--manifold has an associated minimal H-graph and two
   such manifolds have quasi-isometric fundamental groups (in fact,
  bilipschitz equivalent universal covers) if and only
  if their minimal H-graphs are isomorphic. 
\end{th1}

\begin{th2}
  The minimal H-graph associated to a good $3$--manifold is
  balanced. Assuming CCC$_3$, the converse
  is true, namely if a minimal
  H-graph is balanced, then it is the minimal H-graph for some 
  quasi-isometry class of good $3$--manifold groups.
\end{th2}
\begin{th3}
  Assuming CCC$_3$, if two \NAH-manifolds have quasi-isometric fundamental
  groups and their common minimal \NAH-graph is a tree with manifold
  labels then they (and in particular, their fundamental groups) are
  commensurable.
\end{th3}

See Remark~\ref{remark:RFCH} for a
discussion of our use of CCC$_3$. Although CCC would follow from RFCH (residual finiteness conjecture
for word hyperbolic groups), it is not clear how much confidence one
should have in RFCH. We feel that CCC is more plausible, and 
of independent interest. In any case, we expect the conclusions of
the Realization and Commensurability Theorems 
to be true regardless. On the other
hand, we do not know to what extent our restrictive conditions on the
minimal \NAH-graph in the Commensurability Theorem 
are needed; 
whether ``\NAH-manifolds are quasi-isometric if and only
if they are commensurable'' holds in complete generality remains a
very interesting open question. 

\smallskip A slightly surprising byproduct of this investigation are
the \emph{minimal orbifolds} of section \ref{sec:minimal orbifold}, which
play a role similar to commensurator quotients but exist also for
cusped arithmetic hyperbolic orbifolds. Although their existence is
easy to prove, they were new to us.  
Minimal orbifolds are precisely the orbifolds that can
appear as vertex labels of minimal \NAH-graphs.

\section{Minimal orbifolds}\label{sec:minimal orbifold}
We consider only orientable manifolds and orbifolds. Let $N$ be a
hyperbolic $3$--orbifold 
(not necessarily
non-arithmetic) with at least one cusp. Then each cusp of $N$ has a
smallest cover by a toral cusp (one with toral cross section) and this
cover has cyclic covering transformation group $F_C$ of order $1$,
$2$, $3$, $4$, or $6$. We call this order the ``orbifold degree of the
cusp.''

\begin{proposition} Among the orbifolds $N'$ covered by $N$ having the same
  number of cusps as $N$ with each cusp of $N'$ having the same orbifold
  degree as the cusp of $N$ that covers it, there is a unique one,
  $N_0$, that is covered by all the others.  We call this $N_0$ a
  \emph{minimal orbifold}.

  More generally, for each cusp of $N$ specify a ``target'' in
  $\{1,2,3,4,6\}$ that is a multiple of the orbifold degree and ask
  that the corresponding cusp of $N'$ have orbifold degree dividing
  this target. The same conclusion then holds.
\end{proposition}
\begin{proof}

  We first ``neuter'' $N$ by removing disjoint open horoball
  neighborhoods of the cusps to obtain a compact orbifold with
  boundary which we call $N^{0}$. If we prove the proposition for
  $N^{0}$, with boundary components interpreted as cusps, then it
  holds for $N$.

  Let $\widetilde N^{0}$ be the universal cover of $N^{0}$. Any
  boundary component $C$ of $N^{0}$ is isomorphic to the quotient of
  a euclidean plane $\widetilde C$ by the orbifold fundamental group
  $\pi_1^{orb}(C)$, which is an extension of a lattice $\Z^2$ by the
  cyclic group $F_C$ of order $1$, $2$, $3$, $4$, or $6$. We choose an
  oriented foliation of this plane by parallel straight lines and
  consider all images by covering transformations of this foliation on
  boundary planes of $\widetilde N^{0}$.  Each boundary plane of
  $\widetilde N^{0}$ that covers $C$ will then have $|F_C|$ oriented
  foliations, related by the action of the cyclic group $F_C$.  If the
  target degree $n_C$ for the given cusp is a proper multiple of
  $|F_C|$, we also add the foliations obtained by rotating by
  multiples of $2\pi/n_C$.

  We give the boundary planes of $\widetilde N^{0}$ different labels
  according to which boundary component of $N^{0}$ they cover, and we
  construct foliations on them as above.  So each boundary plane of
  $\widetilde N^{0}$ carries a finite number ($1$, $2$, $3$, $4$, or
  $6$) of foliations and the collection of all these oriented
  foliations and the labels on boundary planes are invariant under the
  covering transformations of the covering $\widetilde N^{0}\to
  N^{0}$.  Let $G_N$ be the group of all orientation preserving
  isometries of $\widetilde N^{0}$ that preserve the labels of the
  planes and preserve the collection of oriented foliations. Note that
  $G_N$ does not depend on choices: The only relevant choices are the
  size of the horoballs removed when neutering and the direction of
  the foliation we first chose on a boundary plane of $\widetilde
  N^{0}$. If we change the size of the neutering and rotate the
  direction of the foliation then the size of the neutering and
  direction of the image foliations at all boundary planes with the
  same label change the same amount, so the relevant data are still
  preserved by $G_N$. It is clear that $G_N$ is discrete (this is true
  for any group of isometries of $\H^3$ which maps a set of at least
  three disjoint horoballs to itself).  $N^{0}_0:=\widetilde
  N^{0}/G_N$ is thus an oriented orbifold; it is clearly covered by
  $N^{0}$, has the same number of boundary components as $N^{0}$, and
  the boundary component covered by a boundary component $C$ of
  $N^{0}$ has orbifold degree dividing the chosen target degree
  $n_C$. Moreover if $N^{0}_1$ is any cover of $N^{0}_0$ with the same
  property, then $\widetilde N_1^{0}=\widetilde N^{0}$ and the
  labellings and foliations on $\widetilde N_1^{0}$ and $\widetilde
  N^{0}$ can be chosen the same, so $N^{0}_0$ is the minimal orbifold
  also for $N^{0}_1$. The proposition thus follows.
\end{proof}

Note that a minimal orbifold may have a non-trivial isometry group (in
contrast with commensurator quotients of non-arithmetic hyperbolic
manifolds).

\section{\NAH-graphs}\label{sec:NAHgraph}
The graphs we need will be finite, connected, undirected graphs. We take
the viewpoint that an edge of an undirected graph consists of a
pair of oppositely directed edges. The reversal of a directed edge $e$
will be denoted $\bar e$ and the initial and terminal vertices of
an edge will be denoted $\head{e}$ and $\tail{e}$ ($=\head{\bar e}$).

We will label the vertices of our graph by hyperbolic orbifolds
so we first introduce some terminology for these.
A horosphere section $C$ of a cusp of a hyperbolic orbifold $N$ will
be called the \emph{cusp orbifold}. Although the position of $C$ as a
horosphere section of the cusp involves choice, as a flat
$2$--dimensional orbifold, $C$ is canonically determined up to
similarity by the cusp.  We thus have one cusp orbifold for each cusp
of $N$.

Since a cusp orbifold $C$ is flat, its tangent space $TC$ is
independent of which point on $C$ we choose (up to the action of the
finite cyclic group $F_C$). $TC$ naturally contains the maximal
lattice $\Z^2\subset \pi_1^{orb}(C)$, so it makes sense to talk of a
linear isomorphism between two of these tangent spaces as being
\emph{rational}, i.e., given by a rational matrix with respect to
oriented bases of the underlying integral lattices.  Moreover, such a
rational linear isomorphism will have a well-defined determinant.

\begin{definition}[\NAH-graph]\label{def:NAH-graph}
  An \emph{\NAH-graph} is a finite connected graph with the following data
  labelling its vertices and edges:
\begin{enumerate}
\item Each vertex $v$ is labelled by a hyperbolic orbifold $N_v$ plus
  a map $e\mapsto C_e$ from the set of directed edges $e$ exiting that
  vertex to the set of cusp orbifolds $C_e$ of $N_v$. This map is
  injective, except that an edge which begins and ends at the same
  vertex may have $C_e=C_{\bar e}$.
\item The cusp orbifolds $C_e$ and $C_{\bar e}$  have the same
  orbifold degree.
\item Each directed edge $e$ is labelled by a rational linear
  isomorphism $\ell_e\colon TC_e\to TC_{\bar e}$, with $\ell_{\bar
    e}=\ell_e^{-1}$. Moreover $\ell_e$ reverses orientation (so
  $\det(\ell_e)<0$).
\item\label{NAH-graph(4)} We only need $\ell_e$ up to right
  multiplication by elements of $F_e:=F_{C_e}$; in other words, the
  relevant datum is really the coset $\ell_e F_e$ rather than
  $\ell_e$ itself. This necessitates:
\item $\ell_e$ conjugates the cyclic group $F_e$ to the cyclic
  group $F_{\bar e}$, i.e., $\ell_e F_e =F_{\bar
      e}\ell_e$ (this holds automatically if the orbifold degree
  is $\le2$; otherwise it is equivalent to saying that $\ell_e$
  is a similarity for the euclidean structures).
\item At least one vertex label $N_v$ is a non-arithmetic hyperbolic
  orbifold.
\end{enumerate}
\end{definition}
\begin{definition}\label{def:integral NAH}
We call an \NAH-graph \emph{balanced} if the product of the
determinants of the linear maps on edges around any closed directed
path is equal to $\pm1$.

We call the \NAH-graph \emph{integral} if each $\ell_e$ is an integral
linear isomorphism (i.e., an isomorphism of the underlying
$\Z$--lattices).  Integral clearly implies balanced. 
\end{definition}
\begin{definition}\label{def:assoc int}
  An integral \NAH-graph contains precisely the information to specify
  how to glue neutered versions $N^{0}_v$ of the orbifolds $N_v$
  together along their boundary components to obtain a \NAH-orbifold
  $M$, called the \emph{associated orbifold}. Conversely, any
  \NAH-orbifold $M$ has an \emph{associated integral \NAH-graph}
  $\Gamma(M)$, which encodes its decomposition into geometric pieces.
\end{definition}

\smallskip We next want to define morphisms of \NAH-graphs. 

\begin{definition}[Morphism]\label{def:morphism}
 A \emph{morphism} of
  \NAH-graphs,  $\Gamma\to\Gamma'$, consists of the following data:
\begin{enumerate}
\item\label{data1} an abstract graph homomorphism $\phi\colon\Gamma\to
  \Gamma'$, and 
\item\label{data2} for each vertex $v$ of $\Gamma$, a covering map
  $\pi_v\colon N_v\to N_{\phi(v)}$ of the orbifolds labelling $v$ and
  $\phi(v)$ which respects cusps (so for each
  departing edge $e$ at $v$ one has $\pi_v(C_e)=C_{\phi(e)}$), subject
  to the condition:
\item\label{cond3}
for each directed edge $e$ of $\Gamma$
the following diagram commutes
\begin{equation*}
\xymatrix{
TC_e\ar[d]\ar[r]^{\ell_e}&TC_{\bar e}\ar[d]\\
TC_{\phi(e)}\ar[r]^{\ell_{\phi(e)}}&TC_{\phi(\bar e)}
}  
\end{equation*}
\end{enumerate}
Here the vertical arrows are the induced maps of tangent spaces and
commutativity of the diagram is  up to the indeterminacy of
item \eqref{NAH-graph(4)} of Definition \ref{def:NAH-graph}.
\end{definition}

Morphisms compose in the obvious way, so a morphism is an
\emph{isomorphism} if and only if $\phi$ is a graph isomorphism and
each $\pi_v$ has degree 1.
\begin{definition}
  An \NAH-graph is \emph{minimal} if every morphism to another
  \NAH-graph is an isomorphism.
\end{definition}

The relationship of existence of a morphism between
  \NAH-graphs generates an equivalence relation which we call
  \emph{bisimilarity}. 
\begin{theorem}\label{th:minimal}  
  Every bisimilarity class of \NAH-graphs contains a unique (up to
  isomorphism) minimal member, so every \NAH-graph in the bisimilarity
  class has a morphism to it.
\end{theorem}

\NAH-graphs play an analogous role to the two-colored graphs that we
used in \cite{behrstock-neumann} to classify graph-manifold groups up
to quasi-isometry. For those graphs the morphisms were called ``weak
coverings''---they were color preserving open graph homomorphisms
(i.e., open as maps of $1$--complexes). The term
``bisimilarity'' was coined in that context, since the concept is known to
computer science under this name. Theorem \ref{th:minimal} has a proof
analogous to the proof we gave in \cite{behrstock-neumann} for
two-colored graphs. We give a different proof, which we postpone to
the next section, since it follows naturally from the discussion there
(an analogous proof works in the two-colored graph case).

\begin{theorem}\label{th:balanced}
  If\/ $\Gamma\to\Gamma'$ is a morphism of \NAH-graphs and $\Gamma$ is
  balanced, then so is $\Gamma'$.  In particular, a bisimilarity class
  contains a balanced \NAH-graph if and only if its minimal \NAH-graph
  is balanced.
\end{theorem}
\begin{proof} We will denote
  the negative determinant of the linear map labelling an edge $e$ of an
  \NAH-graph by $\delta_e$, so $\delta_e=\delta_{\bar e}^{-1}>0$.

  Assume $\Gamma$ is balanced. We will show $\Gamma'$ is balanced.  We
  only need to show that the product of the $\delta_e$'s along any
  simple directed cycle in $\Gamma'$ is $1$.  Let $\calC=(e'_0,e'_1,
  \dots, e'_{n-1})$, be such a cycle, so $\tail{e'_i}=\head{e'_{i+1}}$
  for each $i$ (indices modulo $n$). Denote
\begin{equation}
    \label{eq:D}
D:=\prod_{i=0}^{n-1} \delta_{e'_i}\,,
\end{equation}
so we need to prove that $D=1$.

From now on we will restrict attention to one connected component
$\Gamma_0$ of the full inverse image of $\calC$ under the map $\Gamma\to
\Gamma'$.

Let $e$ be an edge of $\Gamma_0$ which maps to an edge $e'$ of $\calC$
($e'$ may be an $e'_i$ or an $\bar e'_i$). Its start and end
correspond to cusp orbifolds which cover
the cusp orbifolds corresponding to start and end of $e'$ with
covering degrees that we shall call $d_e$ and $d_{\bar e}$
respectively. We claim that
  \begin{equation}
    \label{eq:bal}
    d_e\delta_{e'}=\delta_ed_{\bar e}\,.
  \end{equation}
  Indeed, if the cusp orbifolds are all tori, this equation (multiplied
  by $-1$) just
  represents products of determinants for the two ways of going around
  the commutative diagram in part \eqref{cond3} of Definition
  \ref{def:morphism}.  In general, $d_e$ and $d_{\bar e}$ are the
  determinants multiplied by the order of the cyclic group $F_e$, but
  this is the same factor for each of $d_e$ and $d_{\bar e}$ so the
  equation remains correct.

  Let $\widetilde \calC\to \calC$ be the infinite cyclic cover of
  $\calC$, and $\widetilde \Gamma_0\to \Gamma_0$ the pulled back
  infinite cyclic cover of $\Gamma_0$. Define $d_e$ or $\delta_e$ for
  an edge of either of these covers as the value for the image edge.
  
If $(e_1,e_2,\dots,e_m)$ is a directed path in $\widetilde \Gamma_0$
  then the product of the equations \eqref{eq:bal} over all edges of
  this path gives
$\prod_{i=1}^md_{e_i}\delta_{\pi e_i}=\prod_{i=1}^m\delta_{e_i}d_{\bar
  e_i}$, so
\begin{equation*}
  \prod_{i=1}^m\frac{\delta_{\pi
    e_i}}{\delta_{e_i}}=\prod_{i=1}^m\frac{d_{\bar
    e_i}}{d_{e_i}}\,.
\end{equation*}
Since $\widetilde\Gamma_0$ and $\widetilde
\calC$ are both balanced, it follows that $\prod_{i=1}^m\frac{d_{\bar
    e_i}}{d_{e_i}}$ only depends on the start and end vertex of the
path.  Thus, if we fix a base vertex $v_0$ of $\widetilde \Gamma_0$
and define
\begin{equation}
  \label{eq:c(v)}
  c(v):=\prod_{i=1}^m\frac{\delta_{\pi
    e_i}}{\delta_{e_i}}=\prod_{i=1}^m\frac{d_{\bar
    e_i}}{d_{e_i}}\quad\text{for any path from $v_0$ to $v$}\,,
\end{equation}
we get a well defined invariant of the vertices of $\widetilde
\Gamma_0$. For an edge $e$ this
invariant satisfies
$$\frac{c(\head{\bar e})}{c(\head{e})}=
\frac{c(\tail{e})}{c(\head{e})}=\frac{d_{\bar e}}{d_e}\,,$$
whence
$$\zeta(e):=\frac{d_e}{c(\head{e})}$$ 
is an invariant of the undirected edge: $\zeta(e)=\zeta(\bar e)$.

Denote the map $\widetilde \Gamma_0\to \widetilde \calC$ by $\pi$.
Number the vertices sequentially along $\widetilde \calC$ by integers
$i\in \Z$, and for each vertex $v$ of $\widetilde \Gamma_0$ let
$i(v)\in \Z$ be the index of $\pi(v)$. Let
$h\colon\widetilde\Gamma_0\to\widetilde\Gamma_0$ be the covering
transformation: $h(v)$ is the vertex of $\widetilde \Gamma_0$ with the
same image in $\Gamma_0$ as $v$ but with $i(h(v))=i(v)+n$.

Note that the equations \eqref{eq:D} and \eqref{eq:c(v)} and the fact
that $\Gamma_0$ is balanced implies that $c(h(v))=Dc(v)$ for any
vertex $v$ of $\widetilde \Gamma_0$, so
\begin{equation}\label{eq:zeta}
  \zeta(h(e))=D^{-1}\zeta(e)\,.
\end{equation}

From now on consider the edges of $\widetilde \Gamma_0$ directed only
in the direction of increasing $i(v)$. Denote by $Z(j)$ the sum of
$\zeta(e)$ over all edges with $i(\tail{e})=j$.

The sum of the $d_e$'s over outgoing edges at a vertex $v$ of
$\widetilde \Gamma_0$ equals the sum of $d_{\bar e}$'s over incoming
edges at $v$, since each sum equals the degree of the covering map
from the orbifold labelling $v$ to the one labelling vertex $\pi(v)$
of $\widetilde \calC$. Since $\zeta(e)=d_e/c(v)$ for an outgoing edge
and $\zeta(\bar e)=d_{\bar e}/c(v)$ for an incoming one, the sum of
$\zeta(e)$ over outgoing edges at $v$ equals the sum of $\zeta(e)$ for
incoming ones.  Thus $Z(j)$ is the sum of $\zeta(e)$ over all edges
with $i(\head{e})=j$. These are the edges with $i(\tail{e})=j+1$,
so $Z(j)=Z(j+1)$. Thus $Z(j)$ is independent of $j$. Clearly $Z(j)>0$.
Equation \eqref{eq:zeta} implies $Z(j+n)=D^{-1}Z(j)$, so $D=1$, as was
to be proved. (We are grateful to Don Zagier for help with this
proof.)
\end{proof}
We close this section with an observation promised in the introduction.
\begin{proposition}
  An orbifold can be a vertex label in a minimal \NAH-graph if and
  only if it is a minimal orbifold.
\end{proposition}
\begin{proof} We first prove ``only if.''  In condition \eqref{cond3}
  of the definition of a morphism, $\ell_{\phi(e)}$ is determined by
  $\ell_e$ since the vertical arrows are isomorphisms (but $\ell_e$
  may not be determined by $\ell_{\phi(e)}$, since the indeterminacy
  $F_{\phi(e)}$ of $\ell_{\phi(e)}$ may be greater than the
  indeterminacy $F_e$ of $\ell_e$).  Thus for any \NAH-graph there is an
  outgoing morphism to a new \NAH-graph for which the underlying graph
  homomorphism is an isomorphism and every orbifold vertex label is
  simply replaced in the new graph by the corresponding minimal
  orbifold; the new edge labels are then as just described.

  For the converse, if $N$ is a minimal orbifold, any \NAH-graph which
  is star-shaped, with $N$ labelling the middle vertex and with
  one-cusp non-arithmetic commensurator quotients labelling the outer
  ones, is minimal.  One needs a one-cusp commensurator quotient for
  each cusp degree in $\{1,2,3,4,6\}$ for this construction; these are
  not hard to find.
\end{proof}
\section{Minimal \NAH-graphs classify for quasi-isometry} 
\label{sec:nah graphs}

Let $M=M^3$ be a \NAH-manifold. For simplicity of exposition we 
first discuss the case that $M$ has no arithmetic pieces, and then
discuss the modifications needed when arithmetic pieces also occur.

$M$ is then pasted together from pieces $M_i$, each of which is a
neutered non-arithmetic hyperbolic manifold. We can choose neuterings
in a consistent way, by choosing once and for all a neutering for the
commensurator quotient in each commensurability class of
non-arithmetic manifolds, and choosing each $M_i$ to cover one of
these ``standard neutered commensurator quotients.''

The pasting identifies pairs of flat boundary tori with each other by
affine maps which may not be isometries. To obtain a smooth metric on
the result we glue a toral annulus $T^2\times I$ between the two
boundaries with a metric that interpolates between the flat metrics at
the two ends in a standard way (if $g_0$ and $g_1$ are the flat
product metrics induced on $T^2\times I$ by the flat metric on its
left and right ends we use $(1-\rho(t))g_0+\rho(t)g_1$ where
$\rho\colon [0,1]\to[0,1]$ is some fixed smooth bijection with derivatives
$0$ at each end). This specifies the metric on $M$ up to rigid
translations of the gluing maps, so we get a compact family of
different metrics on $M$.

The universal cover $\widetilde M$ is glued from infinitely many
copies of the $\widetilde M_i$'s with ``slabs'' $\R^2\times I$
interpolating between them. Each slab admits a full $\R^2$ of
isometric translations.  The pieces $\widetilde M_i$ will be called
``pieces'' and each boundary component of a piece will be called a
``flat'' and will be oriented as part of the boundary of the piece it
belongs to.

Let $M'$ be another such manifold and $\widetilde M'$ its universal
cover, metrized as above.  Kapovich and Leeb \cite{KapovichLeeb:haken} show that
any quasi-isometry $f\colon \widetilde M\to \widetilde M'$ is a
bounded distance from a quasi-isometry that maps slabs to slabs and geometric pieces to geometric pieces. Then
a theorem of Schwartz \cite{Schwartz:RankOne} says that $f$ is a
uniformly bounded distance from an isometry on each piece $\widetilde
M_i$. Our uniform choice of neuterings assures that we can change $f$
by a bounded amount to be an isometry on each piece and an isometry
followed by a shear map on each slab (a ``shear map'' $\R^2\times I\to
\R^2\times I$ will mean one of the form $(x,t)\mapsto (x+\rho(t)v,t)$ with
$v\in R^2$ and $\rho\colon I\to I$ as described earlier). We then say
$f$ is \emph{straightened}.

Consider now the group\footnote{Note that $\I(\widetilde M)\to
  \mathcal Q\mathcal I(\widetilde M)$ is an isomorphism, where  $
  \mathcal Q\mathcal I$ denotes the group of quasi-isometries (defined
  by identifying maps that differ by a bounded distance).}
$$
\I(\widetilde M):=\{f\colon \widetilde M\to \widetilde M~|~ f\text{ is a
  straightened quasi-isometry}\}\,.
$$ 

$\I(\widetilde M)$ acts on the set of pieces of $\widetilde M$; 
the pieces $\widetilde M_i$ that cover a given piece $M_i$ of $M$ are
all in one orbit of this action, but the orbit may be larger. The
subgroup $\I_{\widetilde M_i}(\widetilde M)$ of $\I(\widetilde M)$
that stabilizes a fixed piece $\widetilde M_i$ of $\widetilde M$ acts
discretely on $\widetilde M_i$, so $\widetilde M_i/\I_{\widetilde
  M_i}(\widetilde M)$ is an orbifold (which clearly is covered by
$M_i$)

The subgroup $\I_S(\widetilde M)\subset \I(\widetilde M)$ that
stabilizes a slab $S\subset \widetilde M$ acts on $S$ by isometries
composed with shear maps. It acts discretely on each boundary
component of $S$ but certainly not on $S$. But an element that is a
finite order rotation on one boundary component must be a similar
rotation on the other boundary component (and is in fact finite order
on the slab; the group $\I_S(\widetilde M)$ is abstractly an extension
of $\Z^2\times\Z^2$ by a finite cyclic group $F_S$ of order $1$, $2$,
$3$, $4$, or $6$; the two $\Z^2$'s are the translation groups
for the two boundaries of $S$).

We form an \NAH-graph $\Omega(\widetilde M)$ as follows. The
vertices of $\Omega(\widetilde M)$ correspond to $\I(\widetilde M)$--orbits of
pieces and the edges correspond to orbits of slabs---the edge
determined by a slab connects the vertices determined by the abutting
pieces.  We label each vertex by the corresponding orbifold
$\widetilde M_i/\I_{\widetilde M_i}(\widetilde M)$ and each edge by
the derivative of the affine map between the flats that bound a
corresponding slab $S$ (this map is determined up to the cyclic group
$F_S$). 

\begin{proposition}\label{prop:th1}
  The \NAH-graph $\Omega(\widetilde M)$ constructed above is
  determined up to isomorphism by $\widetilde M$.  The manifold 
  $\widetilde M$ is determined up to bilipschitz diffeomorphism 
  by $\Omega(\widetilde M)$.
\end{proposition}
\begin{proof} 
  The first sentence of the proposition is true by construction. 

  We describe how to reconstruct $\widetilde M$ from
  $\Omega(\widetilde M)$. $\Omega(\widetilde M)$ gives specifications
  for inductively gluing together pieces that are the universal covers
  of the orbifolds corresponding to its vertices, and slabs between
  these pieces, according to a tree: we start with a piece $\widetilde
  N$ corresponding to a vertex $v$ of $\Gamma$ and glue slabs and
  adjacent pieces on all boundary components as specified by the
  outgoing edges at $v$ in $\Omega(\widetilde M)$, and repeat this
  process for each adjacent piece, and continue inductively. (There is
  an underlying tree for this construction which is the universal
  cover of the graph obtained by replacing each edge of $\Gamma$ by a
  countable infinity of edges.) The construction involves choices,
  since each gluing map is only determined up to a group of isometries
  of the form $\R^2\rtimes F_S$, where $S$ is the slab. We need to
  show that the resulting manifold is well defined up to bilipschitz
  diffeomorphism.

  The constructed manifold $X$ and the original $\widetilde
  M$ can both be constructed in the same way from $\Omega(\widetilde
  M)$, but they potentially differ in the choices just mentioned. We
  can construct a bilipschitz diffeomorphism $f\colon X\to \widetilde
  M$ inductively, starting with an isometry from one piece of $X$ to
  one piece of $\widetilde M$ and extending repeatedly over adjacent
  pieces.  At any point in the induction, when extending to an
  adjacent piece across an adjacent slab $S$, we use an isometry of
  the adjacent piece $X_i$ of $X$ to the adjacent piece $M_i$ of
  $\widetilde M$ that takes the boundary component $X_i\cap S$ of
  $X_i$ to the boundary component $M_i\cap S$ of $M_i$.  Restricted to
  this boundary component $E$, this isometry is well defined up to the
  action of a lattice $\Z^2$, so there is a choice that can be
  extended across the slab $S$ with an amount of shear bounded by the
  diameter of the torus $E/\Z^2$.  Since only finitely many isometry
  classes of such tori occur in the construction, we can inductively
  construct the desired diffeomorphism using a uniformly bounded
  amount of shear on slabs.  This diffeomorphism therefore has a
  uniformly bounded bilipschitz constant, as desired.
\end{proof}

We now describe how the above arguments must be modified if $M$ has
arithmetic pieces.  Let $M_1$ be an arithmetic hyperbolic piece which
is adjacent to a non-arithmetic hyperbolic piece $M_2$. We have
already discussed how $M_2$ is neutered; we take an arbitrary
neutering of $M_1$ (and any other arithmetic pieces) which we will
adjust later. As before, the notation $M_i$ refers to the neutered
pieces and we glue $M$ from these pieces mediating with toral annuli
$T^2\times I$ between them.

We aim to show that we can adjust the neutering of the arithmetic
pieces so that any quasi-isometry of $\widetilde M$ can be
straightened as in the beginning of this section to be an isometry
on pieces and a shear map on slabs.

Consider $\widetilde M_1$ as a subset of $\H^3$ obtained by removing
interiors of infinitely many disjoint horoballs. Schwartz
\cite{Schwartz:RankOne} shows that any quasi-isometry of $\widetilde M_1$
is a bounded distance from an isometry of $\widetilde M_1$ to a
manifold obtained by changing the sizes of the removed horoballs by a
uniformly bounded amount.

In the universal cover $\widetilde M$ choose lifts $\widetilde M_1$
and $\widetilde M_2$ glued to the two sides $\partial_1S$ and
$\partial_2S$ of a slab $S\cong \R^2\times I$. Consider a
quasi-isometry $f$ of $\widetilde M$ which maps $S$ to a bounded
Hausdorff distance from itself. We can assume that $f$ is an isometry
on $\widetilde M_2$. By the previous remarks, the map $f$ restricted
to $\widetilde M_1$ is a bounded distance from an isometry of
$\widetilde M_1$ which moves its boundary component $\partial_1S$ to a
parallel horosphere (if we consider $\widetilde M_1$ as a subset of
$\H^3$); by inverting $f$ if necessary, we can assume the diameter of
the horosphere has not decreased.  Thus $\widetilde M_1$ can be
positioned in $\H^3=\{(z,y)\in \C\times\R:y>0\}$ so that $\partial_1S$
is the horosphere $y=1$ and $f(\partial_1S)$ is the horosphere
$y=\lambda$ for some $\lambda\le 1$. Using a smooth isotopy of $f$
which is supported in an
$\epsilon$--neighborhood of the region between $f(\partial_1S)$ and
$\partial_1S$, and which moves $f(\partial_1S)$ to $\partial_1S$,  we
can adjust $f$ to map $\partial_1S$ to itself. This moves each point
of $f(\partial_1S)$ to its closest point on $\partial _1 S$ by a
euclidean similarity, scaling distance uniformly by a factor of
$\lambda$.  The resulting adjusted $f$ is still a quasi-isometry, so 
restricted to $S$ we then have a quasi-isometry which scales metric on
$\partial_1S$ by $\lambda$ and is an isometry
on $\partial_2S$. This is only possible if $\lambda=1$, so $f$, once
straightened on $\widetilde M_1$, maps $\partial_1 S$ to itself.

Now consider the subgroup of the group of quasi-isometries of
$\widetilde M$ which takes $\widetilde M_1$ to itself, and just
consider its restriction to $\widetilde M_1$, which we can think of as
embedded in $\H^3$. By straightening, we have a group of isometries of
$\H^3$ which preserves a family of disjoint horoballs (the ones that
are bounded by images of $\partial_1S$). Any subgroup of $\Isom(\H^3)$
which preserves an infinite family of disjoint horoballs is discrete.
Thus, from the point of view of the construction above, $M_1$ behaves
like a non-arithmetic piece. By repeating the argument, this behavior
propagates to any adjacent arithmetic pieces, hence, so long as at
least one piece is non-arithmetic, the construction of the graph
$\Omega(\widetilde M)$ goes through as before and the proof of
Proposition extends.

\begin{proposition}\label{prop:minimal}
  $\Omega(\widetilde M)$ is balanced, and is the minimal \NAH-graph in
  the bisimilarity class of the \NAH-graph $\Gamma(M)$ associated with
  $M$ (Definition \ref{def:assoc int}).
\end{proposition}
\begin{proof}
  By construction, there is a morphism $\Gamma(M)\to \Omega(\widetilde
  M)$.  In particular, $\Omega(\widetilde M)$ is in the bisimilarity
  class of $\Gamma(M)$. It is balanced by Theorem \ref{th:balanced},
  since $\Gamma(M)$ is integral. It remains to show that it is the
  minimal \NAH-graph in its class.

We have not yet proved Theorem \ref{th:minimal}, which says that there
is a unique minimal graph in each bisimilarity class. The proof that
$\Omega(\widetilde M)$ is minimal will follows from that proof, 
so we do that first.
\begin{proof}[Proof of Theorem \ref{th:minimal}]
  The construction of the proof of Proposition \ref{prop:th1} works
  for any \NAH-graph $\Gamma$, gluing together infinitely may copies
  of the universal covers $\widetilde N_i$ of the orbifolds that label
  the vertices, with slabs between them, according to an infinite tree
  (the universal cover of the graph obtained by replacing each edge of
  $\Gamma$ by countable-infinitely many). We get a simply connected
  riemannian manifold $X(\Gamma)$ which, as long as $\Gamma$ is 
  finite, is well defined up to
  bilipschitz diffeomorphism by the same argument as before. 

  The group of straightened self-diffeomorphisms $\I(X(\Gamma))$, when
  restricted to a piece $\widetilde N_i$, includes the covering
  transformations for the covering $\widetilde N_i\to N_i$. It follows
  that the construction of a \NAH-graph from $X(\Gamma)$, as given in 
  the first part of this section, yields an \NAH-graph
  $\Omega(X(\Gamma))$ together with a morphism $\Gamma\to
  \Omega(X(\Gamma))$.

  If $\Gamma\to\Gamma'$ is a morphism of \NAH-graphs, then $X(\Gamma)$
  and $X(\Gamma')$ are bilipschitz diffeomorphic, since the
  instructions for assembling them are equivalent.  Hence
  $\Omega(X(\Gamma))=\Omega(X(\Gamma'))$; call this graph $m(\Gamma)$.
  Since $m(\Gamma)=m(\Gamma')$ and the existence of a morphism
  generates the relation of bisimilarity of \NAH-graphs, $m(\Gamma)$ is
  the same for every graph in the bisimilarity class. It is also the
  target of a morphism from every $\Gamma$ in this class. Thus it must
  be the unique minimal element in the class, so Theorem
  \ref{th:minimal} is proved.  \end{proof}

If $M$ is an \NAH-orbifold and $\Gamma=\Gamma(M)$ its associated
integral \NAH-graph, then $X(\Gamma)$ reconstructs $\widetilde M$, so
$\Omega(\widetilde M)=\Omega(X(\Gamma))=m(\Gamma)$, and is hence
minimal by the previous proof, so Proposition \ref{prop:minimal} now
follows.
\end{proof}

\begin{proof}[Proof of Classification Theorem 
    for \NAH-manifolds]
Since $\widetilde M$ is quasi-isometric to $\pi_1(M)$, a
quasi-isometry between fundamental groups of $M$ and $M'$ induces a
quasi-isometry $\widetilde M\to \widetilde M'$. We have already
explained how this can then be straightened and thus give an
isomorphism of the corresponding minimal \NAH-graphs.
\end{proof}
For the Classification Theorem 
to be a complete classification of
quasi-isometry types of fundamental groups of good manifolds we will
need to know that every minimal balanced \NAH-graph is realized by an
\NAH-manifold. We formulate this as a conjecture (the Realization 
Theorem 
says this conjecture follows from the the cusp covering conjecture
CCC$_3$).
\begin{conjecture}
  Every minimal balanced \NAH-graph is the minimal \NAH-graph for some
  \NAH-manifold. Equivalently, every bisimilarity class of \NAH-graphs
  which contains a balanced \NAH-graph contains integral \NAH-graphs
  (Definition \ref{def:integral NAH}).
\end{conjecture}

\section{Covers and RFCH}\label{sec:ccc}

For our realization theorem, and also for our commensurability
theorem, we need to produce covers of $3$--manifolds with prescribed
boundary behavior. The covers we need can be summarized by the
following purely topological conjecture, which we find is of
independent interest.
\begin{ccconjecture}[CCC$_n$]\label{cuspcoveringconj} Let $M$ be a hyperbolic
  $n$-manifold.  Then for each cusp $C$ of $M$ there exists a
  sublattice $\Lambda_C$ of $\pi_1(C)$ such that, for any choice of a
  sublattice $\Lambda'_C\subset \Lambda_C$ for each $C$, there exists
  a finite cover $M'$ of $M$ whose cusps covering each cusp $C$ of $M$
  are the covers determined by $\Lambda'_C$.
\end{ccconjecture}

As we rely on this conjecture for $n=3$, we shall now show that it
follows from the well-known residual finiteness conjecture for
hyperbolic groups (RFCH).

\begin{theorem}\label{RFCHimpliescovers} The residual finiteness
  conjecture for hyperbolic groups (RFCH) implies 
CCC$_n$ for all $n$.
\end{theorem}
\begin{proof}
  The Dehn surgery theorem for relatively hyperbolic groups of Osin
  \cite{osin} and Groves and Manning \cite{groves-manning} guarantees
  the existence of sublattices $\Lambda_C$ of $\pi_1(C)$ for each cusp
  so that, given any subgroups $\Lambda'_C\subset \Lambda_C$
  for each $C$, the result of adding relations to $\pi_1(M)$ which
  kill each $\Lambda'_C$ gives a group $G$ into which the groups
  $\pi_1(C)/\Lambda'_C$ inject and which is relatively hyperbolic
  relative to these subgroups. If the $\Lambda'_C$ are sub-lattices,
  then $G$ is relatively hyperbolic relative to finite (hence
  hyperbolic) subgroups, and is hence itself hyperbolic. By RFCH we
  may assume $G$ is residually finite, so there is a homomorphism of
  $G$ to a finite group $H$ such that each of the finite subgroups
  $\pi_1(C)/\Lambda'_C$ of $G$ injects. The kernel $K$ of the
  composite homomorphism $\pi_1(M)\to G\to H$ thus intersects each
  $\pi_1(C)$ in the subgroup $\Lambda'_C$. The covering of $M$
  determined by $K$ therefore has the desired property.
\end{proof}

Let $\Gamma$ be a \NAH-graph. 
For an edge $e$ of $\Gamma$ let $T_e$ be the tangent space of the
cusp orbifold corresponding to the start of $e$. We can identify
$T_e$ with $T_{\bar e}$ using the linear map $\ell_e$, so we will
generally not distinguish $T_e$ and $T_{\bar e}$.  Then $T_e$
contains two $\Z$--lattices, the underlying lattices for the
orbifolds at the two ends of $e$, and we will denote their
intersection by $\Lambda_e$.  Thus, a torus $T_e/\Lambda$ is a common
cover of the cusp orbifolds at the two ends of $e$ if and only if
$\Lambda\subset \Lambda_e$.

Assuming CCC$_3$, we can now 
choose a sublattice $\Lambda'_e$ of $\Lambda_e$ for each edge
$e$ of $\Gamma$ such that for each vertex $v$ of $\Gamma$ the
corresponding orbifold $N_v$ has a cover $M_v$ with the following
property:
\begin{property}\label{prop:good cover}
  For each cusp of $N_v$ corresponding to an edge $e$ departing $v$,
  all cusps of $M_v$ which cover it are of type $T_e/\Lambda'_e$.
\end{property}

This property is precisely the consequence of the CCC$_3$ that 
we use in our proofs of the Realization and Commensurability 
Theorems.

\begin{remark}\label{remark:RFCH}
  It clearly is desirable to avoid using CCC$_3$, but the current
  state of the art in constructing covers of a hyperbolic manifold $M$
  which restrict to prescribed covers on the cusps does not go far
  enough.  For example, Hempel observed in \cite{hempel} that one can
  do this for characteristic covers of the cusps of degrees avoiding a
  finite set of primes.  A characteristic cover is one given by a
  sublattice of the form $q\pi_1(C)$ for some $q>0$.  This is much too
  restrictive (a cover of the cusps corresponding to ends of an edge
  $e$ in an \NAH-graph can be characteristic for both cusps only if
  map $\ell_e$ is a rational multiple of a $\Z$--isomorphism).  By
  filling a cusp by Dehn surgery and then applying results of
  E.~Hamilton \cite{hamilton} to try to prescribe covers of the
  resulting geodesic one can get a little closer, but still far from
  what is needed.
\end{remark}

\section{Realizing graphs}\label{sec:realizingNAH}
\begin{proof}
  [Proof of Realizability Theorem 
  for \NAH-manifolds] Let $\Gamma$ be a balanced
  \NAH-graph. We want to show there is some \NAH-manifold which
  realizes a graph in its bisimilarity class. As pointed out in the
  previous section, this is equivalent to finding an integral
  \NAH-graph in the bisimilarity class.

  Since we assume CCC$_3$, after we choose a sublattice 
  $\Lambda'_e$ of $\Lambda_e$ for each edge $e$ of $\Gamma$, 
  we may assume Property~\ref{prop:good cover} holds.

  Let $d_v$ be the degree of the cover $M_v\to N_v$.  For an edge $e$
  departing $v$ let $d_e$ be the degree of the corresponding cover of
  cusp orbifolds of $M_v$ and $N_v$, i.e., the index of the lattice
  $\Lambda'_e$ in the fundamental group of the boundary component
  $C_e$ of $N_v$ corresponding to $e$ (this is slightly different from
  the usage in the proof of Theorem \ref{th:balanced}).  Since the
  cusps of the $M_v$ corresponding to the two ends of $e$ are equal,
  we have
  \begin{equation}
    \label{eq:degree}
    d_e\delta_e=d_{\bar e}\,.
  \end{equation}
  
  (Recall that
  $\delta_e$ denotes the determinant of the linear map $\ell_e$.)

  Since $\Gamma$ is balanced, we can assign a positive rational number
  $m(v)$ to each vertex with the property that for any edge $e$ one
  has $m(\tail{e})=\delta_em(\head{e})$. Thus, by \eqref{eq:degree},
  \begin{equation}
    \label{eq:cc}
    \frac{m(\tail{e})}{d_{\bar e}}=\frac{m(\head{e})}{d_e}\,.
  \end{equation}
  Choose a positive integer $b$ such that $n(v):=\frac{bm(v)}{d_v}$ is
  integral for every vertex of $\Gamma$ (so $b$ is some multiple of
  the lcm of the denominators of the numbers $\frac{m(v)}{d_v}$). Let
  $M'_v$ be the disjoint union of $n(v)$ copies of $M_v$, so $M'_v$ is
  a $bm(v)$--fold cover of $N_v$. Let $\pi_v\colon M'_v\to N_v$ be the
  covering map.

  For an edge $e$ of $\Gamma$ from $v=\head{e}$ to $w=\tail e$, the
  number of boundary components of $M'_v$ covering the boundary
  component $C_e$ of $N_v$ corresponding to $e$ is $bm(v)/d_e$. By
  \eqref{eq:cc} this equals $bm(w)/d_{\bar e}$, which is the number of
  boundary components of $M'_w$ covering the boundary component
  $C_{\bar e}$ of $N_w$.  Thus $\pi_v^{-1}C_e$ and $\pi_w^{-1}C_{\bar
    e}$ have the same number of components, and each component is
  $T_e/\Lambda'$, so we can glue $M'_v$ to $M'_w$ along these boundary
  components using any one-one matching between them.  Doing this for
  every edge gives a manifold $M$ whose \NAH-graph has a morphism to
  $\Gamma$; if $M$ is disconnected, replace it by a component (but one
  can always do the construction so that $M$ is connected). This 
  proves the theorem.
\end{proof}

\section{Commensurability}
\begin{proof}[Proof of Commensurability Theorem] 
  Let $M_1$ and $M_2$ be two \NAH-manifolds whose \NAH-graphs are
  bisimilar. Assume that their common minimal \NAH-graph is a tree and
  all the vertex labels are manifolds. We want to show 
  that $M_1$ and $M_2$ are commensurable.

  Let $\Gamma$ be the common minimal \NAH-graph for $M_1$ and $M_2$. 
  Since we assume CCC$_3$, we
  may assume Property~\ref{prop:good cover} holds. 
  Accordingly, for each vertex of $\Gamma$ we may take a common cover 
  $N'_v$ of all the pieces of $M_1$ that cover $N_v$ and then 
  choose the lattices
  $\Lambda'_e$, as in the previous section, to be subgroups of the cusp
  groups of these manifolds $N'_v$. In this way we arrange that the
  pieces $M_v$ of the manifold $M$ constructed in that section are
  covers of $N'_v$, and hence of the pieces of $M_1$. Elementary
  arithmetic as in the proof of the Realization Theorem 
  shows that there is
  a $b_0$ so that if the number $b$ of that proof is a multiple of
  $b_0$ then we can choose the glueing in that proof to make $M$ a
  covering space of $M_1$. Call the resulting manifold $M'_1$. If we
  initially choose $N'_v$ to also cover the type $v$ pieces of $M_2$
  then we can also construct a covering space $M'_2$ of $M_2$ out of
  copies of pieces $M_v$.  The decompositions of $M'_1$ and $M'_2$
  then give \NAH-graphs whose vertices are labelled by $M_v$'s and
  whose edges are labelled by $\Z$--isomorphisms. It suffices to show
  that $M'_1$ and $M'_2$ are commensurable.

  We can encode the information needed to construct $M'_1$ in a
  simplified version $\Gamma_0(M'_1)$ of its \NAH-graph. The
  underlying graph is still just the graph describing the
  decomposition of $M'_1$ into pieces, but the labelling is simplified
  as follows. Each vertex $w$ of $\Gamma_0(M'_1)$ corresponds to a
  copy of some $M_v$, which is a normal covering of the orbifold
  $N_v$. We say vertex $w$ is of \emph{type $v$}. The covering
  transformation group for the covering $M_v\to N_v$ induces a
  permutation group $P_v$ on the edges of $\Gamma_0(M'_1)$ exiting
  $w$.  We record the type and permutation action for each vertex of
  $\Gamma_0(M'_1)$.  It is easy to see that the graph $\Gamma_0(M'_1)$
  with these data records enough information to reconstruct $M'_1$ up
  to diffeomorphism from its pieces. A covering of such a graph
  induces a covering of the same degree for the manifolds they encode.

  A graph with fixed permutation actions at vertices as above is
  called \emph{symmetry restricted} in \cite{neumann}. The graphs
  $\Gamma_0(M'_1)$ and $\Gamma_0(M'_2)$ have the same universal
  covering as symmetry restricted graphs.

  So we would like to know that when two such graphs have a common
  universal covering, then they have a common finite covering.  A
  generalization, proved in \cite{neumann} of Leighton's theorem
  \cite{leighton} (which deals with graphs without the extra
  structure) shows that this is true if the underlying graph is a
  tree, so we are done. 
\end{proof}
The results of \cite{neumann} allow one to carry out the above proof
under slightly weaker assumptions on the minimal graph than being a
tree, but given the other strong assumption (CCC$_3$) used in the proof,
it does not seem worth going into details.

\section{Adding Seifert fibered pieces}\label{sec:seif}

We will  modify Definition~\ref{def:NAH-graph} to allow the inclusion of
Seifert fibered space pieces. The graphs we use are called H--graphs,
and we define them below.

We will need to consider Seifert fibered orbifolds among the pieces,
so we first describe a coarse classification into types. As always,
the manifolds and orbifolds we consider are oriented.  We will
distinguish two types: the oriented Seifert fibered orbifold $N$ is
type ``o'' or ``n'' according as the Seifert fibers can be
consistently oriented or not.
$N$ is type ``o'' if and only if the base orbifold $S$ of the Seifert
fibration is orientable. This can fail in two ways: the topological
surface underlying $S$ may be non-orientable, or $S$ may be
non-orientable because it has mirrors. The latter arises when parts of
$N$ look locally like a Seifert fibered solid torus $D^2\times S^1$
factored by the involution $(z_1,z_2)\mapsto (\overline z_1,\overline
z_2)$ (using coordinates in $\C^2$ with $|z_1|\le 1$ and $|z_2|=1$).
The fibers with $z_1\in \R$ in this local description are intervals
(orbifolds of the form $S^1/(\Z/2)$), and the set of base points of
such interval fibers of $N$ form mirror curves in $S$ which are intervals
and/or circles embedded in the topological boundary of $S$. If any of
these mirror curves are intervals, so they merge with part of the
true boundary of $S$ (image of boundary of $N$), then the
corresponding boundary component of $N$ is a pillow orbifold
(topologically a $2$--sphere, with four $2$--orbifold points).

If two type ``o'' Seifert fibered pieces  in the decomposition of
$M$ are adjacent along a torus, and we have oriented their Seifert
fibers, then there is a sense in which these orientations are
compatible or not. We say the orientation is \emph{compatible} or
\emph{positive} if, when viewing the torus from one side, the
intersection number in the torus of a fiber from the near side with a
fiber from the far side is positive. Note that this is well defined,
since if we view the torus from the other side, both its orientation
and the order of the two curves being intersected have changed, so the
intersection number is unchanged.

We define H--graphs below, with the geometric meanings of the new
ingredients in square brackets.  But there is a caveat to these
descriptions. Just an an \NAH--graph can be associated to the
geometric decomposition of an \NAH--manifold, an H--graph can be
associated with the decomposition of a good manifold or orbifold $M$
(see the Introduction for the definition of ``good''). However, the
\NAH--graph associated with a geometric decomposition is a special
kind of \NAH--graph (it is ``integral'' in the terminology of Section
\ref{sec:NAHgraph}), and an H--graph coming from the geometric
decomposition of a good manifold is similarly special. The geometric
explanations therefore only match precisely for special cases of
H--graphs.

\begin{definition}\label{def:Hgraph}
  An \emph{H--graph} $\Gamma$ is a finite connected graph with decorations on
  its vertices and edges as follows:
  \begin{enumerate}
  \item Vertices are partitioned into two types: \emph{hyperbolic
      vertices} and \emph{Seifert vertices}. The full subgraphs of
    $\Gamma$ determined by the hyperbolic vertices, respectively the
    Seifert vertices, are called the
    \emph{hyperbolic subgraph}, respectively the \emph{Seifert subgraph}.
  \item The hyperbolic subgraph is labelled as in Definition
    \ref{def:NAH-graph}, so that each of its components is an
    \NAH--graph.  In particular, each hyperbolic vertex $v$ is
    labelled by a hyperbolic orbifold $N_v$ and there is a map
    $e\mapsto C_e$ from the set of directed edges $e$ exiting that
    vertex to the set of cusp orbifolds $C_e$ of $N_v$. This map is
    defined on the set of \emph{all} edges exiting $e$, not just the
    edges in the hyperbolic subgraph.  As before, it is injective,
    except that an edge which begins and ends at the same vertex may
    have $C_e=C_{\bar e}$.
  \item Each Seifert vertex is labelled by one of two colors, black or
    white [for an H--graph coming from a geometric decomposition this
    encodes whether the Seifert fibered piece in the geometric
    decomposition of $M$ contains boundary components of the ambient
    $3$--manifold or not, as in \cite{behrstock-neumann}]. 
It is also
    labelled by a Seifert fibration type ``o'' or ``n'', as described
    above. 
  \item For an edge $e$ starting at a Seifert vertex the 
    group $F_e=F_{\overline e}$ is $\{1\}$ or $\{\pm 1\}$. If
    $F_e$ is $\{\pm 1\}$ then the Seifert
    vertex is type ``n''.
  \item\label{it:slope} For each edge $e$ from a hyperbolic vertex to
    a Seifert vertex the group $F_e$ associated to the cusp section
    $C_e$ is either trivial or $\{\pm1\}$. The edge is labelled by a
    non-zero rational vector in $TC_e$ called the \emph{slope} $s_e$,
    determined up to  the action of $F_e$.
    [$s_e$ encodes the direction and length of the fibers of
    the adjacent Seifert fibered piece.]
  \item\label{it:sign}Each edge $e$ connecting a type ``o'' Seifert
    vertex with a type ``o'' Seifert vertex or a hyperbolic vertex
    has a \emph{sign label} $\epsilon_e=\pm1$ with
    $\epsilon_e=\epsilon_{\overline e}$ [this describes compatibility
    of orientations of Seifert fibers of adjacent pieces or---if the
    edge connects a Seifert and a hyperbolic vertex---of Seifert fiber
    and slope].
\item\label{it:eqslope} The data described in items (\ref{it:slope})
  and (\ref{it:sign}) are subject to the equivalence relation
  generated by the following moves:
  \begin{enumerate}
\item \label{itit:c} For any type ``o''
  Seifert vertex the signs at all edges adjacent to it may be
  multiplied by $-1$ [reversal of orientation of the Seifert fibers].
  \item The slope $s_e$ of item (\ref{it:slope}) can be multiplied by
    $-1$ while simultaneously multiplying $\epsilon_e$ by $-1$.
\item\label{it:proj slopes}
  For any Seifert vertex, the slopes  at all
  adjacent hyperbolic vertices may be multiplied by a fixed non-zero
  rational number.
\end{enumerate}
\end{enumerate}
\end{definition}

Note that the data encoded by the sign weights modulo the equivalence
relation of item 
(\ref{itit:c}) 
are equivalent to
an element of $H^1(\Gamma\setminus \Gamma_n,\Gamma_h;\Z/2)$,
where $\Gamma_h$ is the hyperbolic subgraph and $\Gamma_n$ the full
subgraph on Seifert vertices of type ``n''.

We define a \emph{morphism of H--graphs}, $\pi\colon \Gamma\to \Gamma'$, to be
an open graph homomorphism which restricts to a \NAH--graph morphism
of the hyperbolic subgraphs (Definition~\ref{def:morphism}), preserves
the black/white coloring on the Seifert vertices, and on the edges
between hyperbolic and Seifert vertices preserves the slope (in the
sense that the slope at the image edge is the image under the tangent
map on cusp orbifolds of the slope at the source edge). Moreover, it
must map type ``n'' vertices to type ``n'' vertices, 
and when an ``o'' vertex $v$ is mapped
to an ``n'' vertex $w$, then the preimage of each edge at $w$ must
either include edges of different signs, or an edge terminating in a
type ``n'' Seifert vertex.

As in Section~\ref{sec:NAHgraph}, the existence of morphisms between 
H--graphs generates an equivalence relation which we call 
\emph{bisimilarity}. 

\begin{proof}[Proof of Classification Theorem]

  To prove the Classification Theorem 
  we will show that each bisimilarity
  class of H--graphs has a minimal element, and if a H--graph comes
  from a non-geometric manifold $M$ then the minimal H--graph
  determines and is determined by $\widetilde M$ up to quasi-isometry.

We first explain why the qi-type of the universal cover $\widetilde M$
(or equivalently of $\pi_1(M)$) determines a
minimal H--graph $\Omega(\widetilde M)$.

As in Section \ref{sec:nah graphs}, we can straighten any
quasi-isometry $\widetilde M\to\widetilde M'$ 
and assume it takes geometric pieces
to geometric pieces and slabs to slabs. We may also assume it is an
isometry on hyperbolic pieces.  A Seifert piece in $\widetilde M$ is
bi-Lipschitz homeomorphic to a fattened tree times $\R$, so the
fibration by $\R$ fibers is coarsely preserved, and we can straighten
it so it is actually preserved. Moreover, if an adjacent piece is
hyperbolic, then the straightened quasi-isometry is an isometry on the
corresponding flat, so the affine structure on fibers of the Seifert
piece is coarsely preserved, and we can straighten so that it is
actually preserved. However, where a Seifert piece is adjacent to a
Seifert piece the $\R\times\R$ product structure on the corresponding
flat (given by Seifert fibers on the two sides) is coarsely preserved,
but the affine structures on the $\R$ fibers need only be preserved up
to quasi-isometry.

Considering straightened quasi-isometries in the above sense, we
denote again
$$
\I(\widetilde M):=\{f\colon \widetilde M\to \widetilde M~|~ f\text{ is a
  straightened quasi-isometry}\}\,.
$$ 

As in Section \ref{sec:nah graphs}, the underlying graph for our
minimal H--graph $\Omega(\widetilde M)$ has a vertex for each orbit of
the action of $\I(\widetilde M)$ on the set of pieces of $\widetilde
M$ and edge for each orbit of the action on the set of slabs. The
labelling of the hyperbolic subgraph is as before; in particular, any
vertex corresponding to an orbit of hyperbolic pieces is labelled by
the hyperbolic orbifold obtained by quotienting a representative piece
in the orbit by its isotropy subgroup in $\I(\widetilde M)$. A Seifert
vertex of the H--graph is of type ``n'' if some element of
$\I(\widetilde M)$ takes a corresponding Seifert fibered piece to
itself reversing orientations of fibers, and is otherwise of type
``o''.  For each type ``o'' vertex we choose an orientation of the
fibers of one piece in the corresponding orbit and then extend
equivariantly to the other pieces in the orbit.

For a Seifert vertex adjacent to at least one hyperbolic vertex the
fibers of the corresponding pieces in $\widetilde M$ carry an affine
structure which is defined up to affine scaling. We choose a specific
scale for each such vertex, so we can speak of length along fibers,
and then the \emph{slope} of item (\ref{it:slope}) of Definition
\ref{def:Hgraph} is given by a tangent vector of unit length, viewed
in the adjacent cusp.

The sign weights of item (\ref{it:sign}) of Definition~\ref{def:Hgraph}
are then defined, and item (\ref{it:eqslope}) of that definition
reflects the choices of orientation and scale which were made.

By construction, the isomorphism type of $\Omega(\widetilde M)$ 
is determined by
$\widetilde M$ and thus two manifolds with quasi-isometric 
fundamental group have the same associated graph.  
It remains to show that the isomorphism type of $\Omega(\widetilde M)$
determines the bilipschitz homeomorphism type of $\widetilde M$. 

Construct a labelled graph $\wt\Omega$ from $\Omega=\Omega(\wt M)$
by first replacing each edge of $\Omega$ by infinitely many edges,
keeping the weights on edges, but adding sign weights to edges at
type ``n'' vertices with infinitely many of each sign, and then taking the universal cover of the
resulting weighted graph.  Finally, 
``o'' and ``n'' labels are now irrelevant and can be removed.

To associate a manifold $X=X(\wt\Omega)$ to this labelled graph, we
must glue together appropriate pieces according to the tree
$\wt\Omega$, with appropriate choices for the gluing between
slabs. The pieces for hyperbolic vertices will be universal covers of
the hyperbolic orbifolds which label them, while for a Seifert vertex
we take the universal cover of some fixed Seifert fibered manifold
with base of hyperbolic type and having a boundary component for each
incident edge in $\Omega$ and---if the vertex is a black vertex---an
additional boundary component to contribute to boundary of $X$. (The
universal cover $Y$ of this Seifert piece is then a fattened tree
times $\R$, and, as in \cite{behrstock-neumann}, it is in fact only
important that there is a bound $B$ such that for each boundary
component of $Y$ there are boundary components of all ``types'' within
distance $B$ of the give boundary component.)

The choices in gluing depend on the types of the abutting pieces:
Between hyperbolic pieces, the gluing map is, as before, determined up
to a group of isometries of the form $\R^2\rtimes F_S$, where $S$ is
the slab. Between Seifert fibered pieces the gluing will be an affine
map such that the fibers from the two pieces then intersect in the
intervening flat with sign given by the sign label of the
edge. Finally, between a Seifert fibered and a hyperbolic piece the
gluing will be an affine map matching unit tangent vector
along fibers with the slope vector for the hyperbolic piece. For each
edge of $\Omega$ we make a fixed choice of how to do the gluing
subject to the above constraints and do it this way for every
corresponding edge of $\wt\Omega$.

To complete the proof, it remains to show that, independent of these
choices, there exists a bilipschitz homeomorphism from $\wt M$ to $X$.

As in the proof of Proposition~\ref{prop:th1}, the desired bilipschitz
homeomorphism is built inductively, starting with a homeomorphism from
one piece of $\wt M$ to a piece of $X$ and then extending via adjacent
slabs to adjacent pieces. There are four cases: (1), when both
adjacent pieces are hyperbolic, this is exactly the case of
Proposition~\ref{prop:th1}; (2), when both pieces are Seifert fibered;
(3), extending from a hyperbolic piece to an adjacent Seifert fibered
piece; and (4), extending from a Seifert fibered piece to an adjacent
hyperbolic piece.
For Case (2) we use \cite[Theorem 1.3]{behrstock-neumann} (as in the
proof of \cite[Theorem 3.2]{behrstock-neumann}) to extend over the
adjacent Seifert fibered piece, respecting the ``types'' of boundary
components (i.e., belonging to boundary of $M$ or not, and if not,
then the ``type'' is given by the edge of $\Omega$ that the boundary
component corresponds to). Case (3) is essentially the same argument,
and Case (4) is immediate.

Thus we obtain the desired bilipschitz homeomorphism, completing the
proof.
\end{proof}

\begin{proof}[Proof of Realization Theorem] 
    The construction of $\Omega(\widetilde M)$, given above, has built
    in a morphism from the H-graph associated to $M$.  Since the graph
    associated to $M$ is balanced, it follows from
    Proposition~\ref{prop:minimal} that $\Omega(\widetilde M)$ is
    balanced.
   
    The balanced condition is only a constraint on NAH-graph
    components of an H-graph.  Thus, from the case of NAH-graphs 
    which we established in Section~\ref{sec:realizingNAH}, we may
    conclude that, assuming $CCC_{3}$, every balanced minimal H-graph
    is the minimal H-graph of some quasi-isometry class of good
    $3$--manifold group.
\end{proof}

\Notes

\endNotes

\end{document}